\documentclass[a4paper,10pt,abstract=on]{scrartcl}

\usepackage{sidecap}
\usepackage[utf8]{inputenc}
\usepackage[T1]{fontenc}
\usepackage[english]{babel}
\usepackage{hyphenat}
\usepackage{amsmath,amssymb,amsfonts}
\usepackage{amsthm}
\usepackage[svgnames]{xcolor}
\usepackage{mathtools}
\mathtoolsset{showonlyrefs}
\usepackage{geometry}
\usepackage{bm}
\usepackage[shortlabels]{enumitem}
\usepackage{csquotes}
\usepackage{booktabs}
\usepackage{cite}
\usepackage{subcaption}
\usepackage[hidelinks]{hyperref}
\usepackage{multirow}
\usepackage{multicol}

\usepackage[svgnames]{xcolor}
\usepackage{mathtools}
\usepackage{bm}
\usepackage{csquotes}
\usepackage{booktabs}
\mathtoolsset{showonlyrefs}
\usepackage[shortlabels]{enumitem}
\usepackage[hidelinks]{hyperref}
\usepackage{comment}
\usepackage{todonotes}

\theoremstyle{plain}
\newtheorem{theorem}{Theorem} 
\newtheorem{lemma}[theorem]{Lemma}
\newtheorem{proposition}[theorem]{Proposition}

\theoremstyle{definition}

\newtheorem{algorithm}[theorem]{Algorithm}

\theoremstyle{remark}
\newtheorem{remark}[theorem]{Remark}

\providecommand{\keywords}[1]{
  \small	
  \textbf{\textit{Keywords---}} #1}

\providecommand{\msc}[1]{
  \small	
  \textbf{\textit{2020 AMS Mathematics Subject Classification---}} #1}

\DeclareMathOperator*{\argmin}{arg\,min}

\DeclareMathOperator*{\prox}{prox}
\DeclareMathOperator*{\proj}{proj}

\newcommand{\TV}{\text{TV}}

\newcommand{\N}{\mathbb N}

\newcommand{\R}{\mathbb R}

\newcommand{\Fun}[1]{\mathcal{#1}}

\newcommand{\FK}{\Fun K}

\newcommand{\FL}{\Fun L}

\newcommand{\Mat}[1]{\bm{#1}}

\newcommand{\MI}{\Mat I}

\newcommand{\MQ}{\Mat Q}

\newcommand{\Mzero}{\Mat 0}
\newcommand{\Vek}[1]{\bm{#1}}

\newcommand{\Vx}{\Vek x}
\newcommand{\Vxi}{\Vek \xi}
\newcommand{\Vy}{\Vek y}
\newcommand{\Vz}{\Vek z}

\newcommand{\Vu}{\Vek u}

\newcommand{\Vell}{\Vek \ell}

\newcommand{\sphere}{\mathbb S}
\newcommand{\ball}{\mathbb B}

\newcommand{\tT}{\mathrm{T}}

\begin{document}

\title{Denoising Sphere-Valued Data 
by Relaxed Total Variation Regularization}

\author{Robert Beinert\thanks{R. Beinert is with the Institute of Mathematics,
	Technische Universit\"at Berlin, Stra\ss{}e des 17. Juni 136,
        10623 Berlin, Germany.}
        \qquad
        Jonas Bresch\thanks{J. Bresch is with the Institute of Mathematics,
	Technische Universit\"at Berlin, Stra\ss{}e des 17. Juni 136,
        10623 Berlin, Germany.}
}

\maketitle

\begin{abstract}
    Circle- and sphere-valued data play a significant role
    in inverse problems like 
    magnetic resonance phase imaging and radar interferometry,
    in the analysis of directional information,
    and in color restoration tasks.
    In this paper,
    we aim to restore $(d-1)$-sphere-valued signals
    exploiting the classical anisotropic total variation 
    on the surrounding $d$-dimensional Euclidean space.
    For this, 
    we propose a novel variational formulation,
    whose data fidelity is based on inner products 
    instead of the usually employed squared norms.
    Convexifying the resulting non-convex problem
    and using ADMM,
    we derive an efficient and fast numerical denoiser.
    In the special case of binary (0-sphere-valued) signals,
    the relaxation is provable tight,
    i.e.\ the relaxed solution can be used to construct
    a solution of the original non-convex problem.
    Moreover,
    the tightness can be numerically observed for
    barcode and QR code denoising
    as well as 
    in higher dimensional experiments like
    the color restoration using hue and chromaticity
    and the recovery of SO(3)-valued signals.
\end{abstract}


\hspace{0.25cm}

\msc{94A08, 94A12, 65J22, 90C22, 90C25}

% ------------------------------------------------------------------
\section{Introduction}
\label{sec:intro}
% ------------------------------------------------------------------

With the development of modern acquisition devices,
manifold-valued data arise in an increasing number
of real-world inverse problems.
For instance,
circle-valued data appear in
color restoration in HSV or LCh spaces \cite{NikSte14},
magnetic resonance phase imaging \cite{LEHS2008},
and radar interferometry \cite{BRFa2000}.
Moreover,
sphere-valued data occur in the analysis of directional information \cite{ASWK1993}
and restoration tasks in the chromaticity-brightness setting \cite{PPS2017,QKL2010}.
Since the available measurements are usually corrupted by noise,
the denoising of the considered signals plays a major role in these applications.
For this reason,
the well-established total variation (TV) has been generalized
to the circle-, sphere-, and, more generally, manifold-valued setting
using lifting procedures \cite{CS13,LSKC13},
exploiting the geodesic distance \cite{BBSW16,WDS14,LNPS2017,BerLauSteWei18},
or employing optimal transport (OT) \cite{Ken23}.
Another denoising approach is based on so-called half-quadratic minimization models \cite{BerChaHiePerSte16,GS14}.
In general, 
the lifting ideas drastically increase the dimension,
the OT methods rely on signals on trees,
whereas
the convergence theory behind the other methods is usually based on Hadamard manifolds, 
which excludes circle- and sphere-valued data.

In this paper,
we consider signals
that are valued in the $(d-1)$-sphere $\sphere_{d-1} \coloneqq \{\Vxi \in \R^d : \|\Vxi\|_2 = 1\}$
and that are supported on a graph.
More precisely,
we consider a connected, undirected graph $G =(V, E)$,
where $V \coloneqq \{1,\dots,N\}$ denotes the set of vertices
and $E \coloneqq \{(n,m) : n < m\} \subset V \times V$  the set of edges,
which encodes the data structure. 
The number of edges is henceforth denoted by $M \coloneqq |E|$. 
We are now interested in the restoration 
of a signal $\Vx \coloneqq (\Vx_n)_{n \in V}$ 
with $\Vx_n \in \sphere_{d-1} \subset \R^d$
from disturbed signal values $\Vy = (\Vy_n)_{n \in V}$ 
with $\Vy_n \in \sphere_{d-1}$ 
or, more generally, $\Vy_n \in \R^{d}$.
This kind of denoising problem has, for instance, been considered
in \cite{BeBrSt23,condat_1D2D},
where a convex relaxation of the corresponding Tikhonov regularization in $\R^d$ is proposed.
The details of this approach are briefly discussed in Sec.~\ref{sec:tikh}.
In the same style,
we propose a convex relaxation of the corresponding anisotropic, first-order TV regularization in $\R^d$
to restore $\Vx$ from $\Vy$, see Sec.~\ref{sec:tikh_tv}.
In the special case $d = 1$,
where the problem is reduced to the recovery of binary signals,
the tightness of our relaxation can be proven,
see Theorem~\ref{tho:binary_prob_tightness_tv}.
For similar binary restoration techniques \cite{CHOGENOB2011},
this property is well-known.
In Sec.~\ref{sec:NumericalResults},
we derive a efficient numerical algorithm to solve our relaxation
based on the Alternating Direction Method of Multipliers (ADMM) \cite{bauschke}
and the fast TV program in \cite{Con12v2,Con13v4}.
In the numerical examples,
we apply the resulting algorithm 
to denoise barcodes, QR codes, and SO(3)-valued signals
as well as 
to restore the hue and chromaticity of color images.
Notably,
although the tightness cannot be proven for higher dimensions,
it can be observed numerically.

% ------------------------------------------------------------------
\section{Tikhonov Regularization for Sphere-Valued Data} 
\label{sec:tikh}
% ------------------------------------------------------------------

To recover the sphere-valued signal $\Vx \coloneqq (\Vx_n)_{n \in V}$
on the graph $G = (V,E)$
from the noisy measurement $\Vy \coloneqq (\Vy_n)_{n \in V}$
as introduced in Sec.~\ref{sec:intro}, 
the authors of \cite{BeBrSt23, condat_1D2D} propose 
to convexify the classical, non-convex Tikhonov regularization:
\begin{equation} 
    \label{eq:class-tik}
    \argmin_{\Vx \in \sphere_{d-1}^N} 
    \,\frac{1}{2} \sum_{n \in V} 
    \,\| \Vx_n - \Vy_n \|^2_2
    + \frac{\lambda}{2} \sum_{(n,m) \in E} 
    \,\| \Vx_n - \Vx_m \|^2_2,
\end{equation}
where $\lambda > 0$ is the regularization parameter.
Since $\Vx_n \in \sphere_{d-1}$ is sphere-valued,
and since $\Vy_n \in \R^d$ is fixed,
the additive terms of the objective are essentially given by
\begin{align}
\label{eq:rewriting_L2}
    \|\Vx_n - \Vy_n\|_2^2 
    = - 2\langle \Vx_n,\Vy_n\rangle + \text{const}
    \qquad\text{and}\qquad
    \|\Vx_n - \Vx_m\|_2^2 
    = - 2\langle \Vx_n,\Vx_m\rangle + \text{const}.
\end{align}
Introducing the auxiliary variable 
$\Vell \coloneqq (\Vell_{(n,m)})_{(n,m) \in E} \in \R^M$
and the new objective
$\FL: \R^{d\times N} \times \R^M \mapsto \R$
given by
\begin{equation*}
\label{eq:func}
    \FL(\Vx,\Vell) 
    \coloneqq
    - \sum_{n \in V}
    \langle\Vx_n,\Vy_n\rangle 
    - \lambda\sum_{(n,m) \in E} 
    \, \Vell_{(n,m)},
\end{equation*}
we rewrite the original Tikhonov regularization \eqref{eq:class-tik} as
\begin{equation} 
    \label{eq:class-tik-rewitten}
    \argmin_{\Vx \in \sphere_{d-1}^N, \Vell \in \R^M}
    \quad 
    \FL(\Vx,\Vell)
    \quad 
    \text{s.t.}
    \quad 
    \Vell_{(n,m)} = \langle \Vx_n,\Vx_m\rangle
    \quad
    \text{for all}
    \quad
    (n,m) \in E.
\end{equation}
The crucial idea in \cite{BeBrSt23, condat_1D2D} is 
to characterize the non-convex domain and the equation constraints
using positive semi-definite matrices.

\begin{proposition}[\hspace{-0.5pt}{\cite[Lem~7]{BeBrSt23}}] 
    \label{prop:MatRepresConvRelax}
    It holds $\Vx_n, \Vx_m \in \sphere_{d-1}$ 
    and $\Vell_{(n,m)} = \langle \Vx_n, \Vx_m\rangle$
    if and only if 
    \begin{align*}
    \MQ_{(n,m)} \coloneqq \left[\begin{smallmatrix}
        \MI_d & \Vx_n & \Vx_m \\
        \Vx_n^{\tT} & 1 & \Vell_{(n,m)} \\
        \Vx_m^{\tT} & \Vell_{(n,m)} & 1
    \end{smallmatrix}\right] \in \R^{d+2\times d+2}
    \end{align*}
    is positive semi-definite and has rank $d$.
\end{proposition}

Applying Prop.~\ref{prop:MatRepresConvRelax}
to \eqref{eq:class-tik-rewitten}
and neglecting the rank constraints
yields the convexified regularization:
\begin{align}    
\label{eq:conv-real-tik-sd}
&    \argmin_{\Vx \in \mathbb R^{d\times N}, \Vell \in \mathbb R^M} 
    \quad \FL(\Vx, \Vell) \quad 
    \text{s.t.} 
    \quad
    \MQ_{(n,m)} \succcurlyeq 0
    \quad
    \text{for all}
    \quad
    (n,m) \in E.
\end{align}
Figuratively,
the inner products
$\langle \Vx_n, \Vy_n \rangle$ in $\FL$
push the minimizer in the direction of the data
whereas
the positive semi-definite matrices incorporate the convexified constraints.
Notice that 
the solution $(\Vx^*,\Vell^*)$ of $\eqref{eq:conv-real-tik-sd}$ 
solves \eqref{eq:class-tik}
if and only if all related matrices $\MQ_{(n,m)}$ have rank $d$.

\section{Total Variation Regularization of Sphere-Valued Data}
\label{sec:tikh_tv}

In contrast to the Tikhonov regularization in the previous section,
which is suitable for smooth signals,
we want to recover piecewise constant signals
from noisy measurements.
For this reason,
we replace the squared 2-norm of the regularizer in \eqref{eq:class-tik}
by the 1-norm 
yielding the total variation (TV) regularization:
\begin{equation}    
\label{eq:nonsmooth_tik}
    \argmin_{\Vx \in \sphere_{d-1}^N}
    \quad
    \frac{1}{2} \sum_{n \in V} \|\Vx_n - \Vy_n\|_2^2
    + \lambda \sum_{(n,m) \in E} \|\Vx_n - \Vx_m\|_1.
\end{equation}
Using \eqref{eq:rewriting_L2} again,
we rewrite the objective of \eqref{eq:nonsmooth_tik} into
\begin{align*}      
    \FK(\Vx) \coloneqq 
    -\sum_{n \in V} \langle\Vx_n, \Vy_n\rangle
    + \lambda\TV(\Vx)
    \qquad\text{with}\qquad
    \TV(\Vx) 
    \coloneqq
    \sum_{(n,m) \in E} \|\Vx_n - \Vx_m\|_1.
\end{align*}
Convexifying the sphere-valued domain $\sphere_{d-1}^N$,
we propose to solve the following convex minimization problem:
\begin{align}
    & \argmin_{\Vx \in \R^{d\times N}}
    \quad
    \FK(\Vx)
    \quad
    \text{s.t.}
    \quad
    \Vx_n \in \ball_d
    \quad
    \text{for all}
    \quad
    n \in V,\label{eq:nonsmooth_tik_simp}
\end{align}
where $\ball_{d} \coloneqq \{\Vxi \in \R^d : \|\Vxi\|_2 \leq 1\}$.
Similarly to Sec.~\ref{sec:tikh},
the inner products $\langle \Vx_n, \Vy_n \rangle$ in the loss 
heuristically push the solution in direction of the given data
and to the boundary of the balls.
The embedding into the Euclidean vector space
and the reformulated new objective $\FK$
are the main differences 
to other denoising models like in \cite{CHOGENOB2011,Con13v4,BerLauSteWei18}.
The convexification can be also obtained using Prop.~\ref{prop:MatRepresConvRelax} and a Schur complement argument.

\paragraph{Tightness of the Convexification}
\label{sec:tightness_S0}

For the one-dimensional setup,
the relaxed problem \eqref{eq:nonsmooth_tik_simp}
actually becomes a tight convexification of the original formulation \eqref{eq:nonsmooth_tik},
i.e.\ having a solution of \eqref{eq:nonsmooth_tik_simp},
we can always construct a solution of \eqref{eq:nonsmooth_tik}.
For this,
we introduce the characteristic $\chi_\eta$ 
regarding the level $\eta \in [-1,1]$
of a signal $\Vx \coloneqq (\Vx_n)_{n\in V} \in \ball_1^N$ by
\begin{align*}
    \chi_{\eta}(\Vx) 
    \coloneqq
    (\chi_\eta(\Vx_n))_{n \in V}
    \qquad\text{with}\qquad
    \chi_{\eta}(\Vx_n)
    \coloneqq
    \begin{cases}
    1 & \text{if } \Vx_n > \eta, \\
    -1 & \text{if } \Vx_n \le \eta. \\
    \end{cases}
\end{align*}
For any $\Vx \coloneqq (\Vx_n)_{n\in V} \in \ball_1^N$,
the characteristic $\chi_\eta(\Vx)$ is a binary signal,
i.e.\ $\chi_\eta(\Vx) \in \sphere_0^N$.
Moreover, 
the characteristic can be used 
to obtain an integral representation of absolute differences 
coming from the TV regularizer .

\begin{lemma}
\label{lem:coarea}
    For $\Vx_n, \Vx_m \in \ball_1$,
    it holds
    \begin{equation*}
        |\Vx_n - \Vx_m| 
        = \frac{1}{2}\int_{-1}^1 |\chi_\eta(\Vx_n) - \chi_\eta(\Vx_m)| \;\mathrm{d}\eta.    
    \end{equation*}
\end{lemma}

\begin{proof}
    Without loss of generality,
    we assume $\Vx_m < \Vx_n$
    and obtain
    \begin{equation*}
        \int_{-1}^1 |\chi_\eta(\Vx_n) - \chi_\eta(\Vx_m)| \;\mathrm{d}\eta
        =
        \int_{\Vx_m}^{\Vx_n} |1 - (-1)| \;\mathrm{d}\eta 
        = 2 \, |\Vx_m - \Vx_n|.
        \qedhere
    \end{equation*}
\end{proof}

Note that Lem.~\ref{lem:coarea} can be interpreted as
discrete version of the well-known coarea formula \cite{Fed59,FleRis60}.

\begin{theorem}
\label{tho:binary_prob_tightness_tv}
    Let $\Vx^* \in \ball_1^N$
    be a solution of \eqref{eq:nonsmooth_tik_simp} for $d=1$.
    Then $\chi_\eta(\Vx^*)$ is a solution of \eqref{eq:nonsmooth_tik}
    for almost all $\eta \in [-1,1]$.
\end{theorem}

\begin{proof}
    The proof follows ideas from \cite{CHOGENOB2011}. 
    For any $\Vx_n \in \ball_1$,
    we have the integral representation:
    \begin{equation*}
        \frac{1}{2}\int_{-1}^1 \chi_\eta(\Vx_n) \mathrm{d}\eta 
        = \frac{1}{2}\int_{-1}^{\Vx_n}1\;\mathrm{d}\eta + \frac{1}{2}\int_{\Vx_n}^1 (-1)\;\mathrm{d}\eta 
        = \Vx_n.
    \end{equation*}
    Together with Lem.~\ref{lem:coarea},
    this yields
    \begin{align*}
        \FK (\Vx)  
        &=
        - \sum_{n \in V} \, \Vx_n\Vy_n 
        + \lambda\TV(\Vx)
        = - \sum_{n \in V} \, \Vx_n\Vy_n 
        + \lambda \sum_{(n,m) \in E} |\Vx_n - \Vx_m|     
        \\
        & = \frac{1}{2}
        \int_{-1}^1 \bigg[ -\sum_{n \in V} \, \chi_\eta(\Vx_n)\Vy_n
        + \lambda \sum_{(n,m) \in E} \, |\chi_\eta(\Vx_n) - \chi_\eta(\Vx_m)| \bigg] \mathrm{d}\eta
        = \frac{1}{2}\int_{-1}^1 \FK(\chi_\eta(\Vx))\;\mathrm{d}\eta.
    \end{align*}
    For a solution $\Vx^*$ of \eqref{eq:nonsmooth_tik_simp},
    we particularly have $\FK(\Vx^*) 
    = \frac{1}{2}\int_{-1}^1
    \FK(\chi_\eta(\Vx^*))\;\mathrm{d}\eta$.
    Since $\FK(\Vx^*)$ is the minimal objective value,
    the integrand on the right side attends the minimum
    for almost all $\eta \in [-1,1]$ as well.
\end{proof}

Unfortunately, 
the tightness of the relaxation cannot be immediately extended to higher dimensions ($d>1$).
The main obstacle is the generalization of $\chi_\eta$ 
and of the related integral representations.

\section{Numerical Results}
\label{sec:NumericalResults}

The relaxed formulation \eqref{eq:nonsmooth_tik_simp} may be solved 
using any appropriate convex solver. 
To obtain a numerically efficient algorithm,
we rely on the so-called Alternating Direction Method of Multipliers (ADMM) \cite{bauschke}.
In more detail,
we consider the splitting $\FK(\Vx) + \iota_{\ball_d^N}(\Vu)$ with $\Vx - \Vu = \Mzero$,
where $\iota_{\ball_d^N}(\Vx) = 0$ for $\Vx \in \ball_d^N$ 
and $\iota_{\ball_d^N}(\Vx) = +\infty$ otherwise.
For our specific setting, 
ADMM reads as Alg.~\ref{alg:1},
where the proximation of $\TV$ is defined as
$\smash{\prox_{\TV,\gamma}(\Vz) 
    \coloneqq 
    \argmin_{\Vx \in \R^{d \times N}} 
    \bigl\{
    \TV(\Vx) + \frac{1}{2\gamma} \sum_{n \in V} \|\Vx_n - \Vz_n\|_2^2\bigr\}}$,
and $\proj_{\ball_d^N}$ denotes to orthogonal projection onto $\ball_d^N$.

\begin{algorithm}
ADMM to solve \eqref{eq:nonsmooth_tik_simp}.
    \label{alg:1}
    Choose $\Vx^{(0)}  = \Vu^{(0)} = \Vz^{(0)} = \Mzero \in \R^{d\times N}$, 
    step size $\rho > 0$
    and TV parameter $\lambda > 0$.
    \\
    \textbf{For} $i \in \N$ \textbf{do:}\\
    \hspace*{0.5cm} $\Vx^{(i+1)} 
    =
    \prox_{\FK,\frac{1}{\rho}}(\Vu^{(i)} - \Vz^{(i)})
    =
    \prox_{\TV, \frac{\lambda}{\rho}}(\Vu^{(i)} - \Vz^{(i)} + \Vy\rho^{-1})$, \\
    \hspace*{0.5cm} $\Vu^{(i+1)} 
    = 
    \prox_{\iota_{\ball_d^N}}(\Vx^{(i+1)} + \Vz^{(i)})
    = \proj_{\ball_d^N}(\Vx^{(i+1)} + \Vz^{(i)})$, \\
    \hspace*{0.5cm} $\Vz^{(i+1)} 
    = \Vz^{(i)} + \Vx^{(i+1)} - \Vu^{(i+1)}$.
\end{algorithm}

\begin{remark}
    Notice that we rely on an anisotropic TV regularization.
    Therefore, 
    we can apply the fast TV programs \cite{Con12v2, Con13v4} coordinatewise
    to compute $\prox_{\TV, \frac\lambda\rho}$ efficiently.
    The convergence of Alg.~\ref{alg:1} is ensured 
    by \cite[Cor.~28.3]{bauschke}.
\end{remark}

The employed algorithm is implemented\footnote{
The code is available at GitHub \url{https://github.com/JJEWBresch/relaxed_tikhonov_regularization}.} 
in Python~3.11.4 using Numpy~1.25.0 and Scipy~1.11.1.
The following experiments are performed on an off-the-shelf iMac 2020
with Apple M1 Chip (8‑Core CPU, 3.2~GHz) and 8~GB RAM.
In all comparisons,
we apply Alg.~\ref{alg:1} to a range of regularization parameters 
and finally choose the parameter 
yielding the smallest mean squared error (MSE).

\paragraph{Binary Signal Denoising}

Typical real-world examples of binary signals are
barcodes and QR codes.
In this example,
the bars are $\sphere_0$-valued,
consist of $5$ or $10 \times 10$ pixels,
and are randomly generated
according to independent Rademacher distribution.
The obtained signals are afterwards pixelwise disturbed 
by additive Gaussian noise,
whose standard deviation is denoted by $\sigma$.
To denoise the synthetic data,
we apply Alg.~\ref{alg:1}
and compare the results with
(i) the anisotropic TV denoising of barcodes and QR codes by Choksi et al. (ANISO-TV) \cite{CHOGENOB2011}
and
(ii) the fast TV program by Condat (fast-TV) \cite{Con12v2,Con13v4}.
Notice that our method and ANISO-TV are tight 
such that the characteristic of the numerical solution 
almost surely solves \eqref{eq:nonsmooth_tik}.
In the experiments, 
we employ $\chi_0$.
In difference, 
fast-TV is a state-of-the-art unconstraint denoiser 
providing a solution in $\R^N$,
which we afterwards project to $\sphere_0^N$.
This corresponds to the first iteration of our method.
A qualitative comparison of the considered methods is given in 
Figure~\ref{fig:S0_1D_comp_ADMM_ANISO} and \ref{fig:S0-2D}.
A quantitative study is reported in 
Table~\ref{table:S0-signal} and \ref{table:S0-image},
where we compare the Mean Square Error (MSE)
and the Mean Intersection over Union (MIoU)
averaged over 50 randomly generated datasets.
The MIoU here corresponds to the ratio of
correctly recovered pixels
and the total number of pixels.
In summary,
we outperform ANISO-TV qualitative as well as quantitative.
Moreover, 
our method here already yields $\sphere_0^N$ solutions
making the final projection step obsolete---%
in contrast to ANISO-TV.
We observe that fast-TV here yields a rapid and accurate heuristic.

\begin{figure}[t!]
\includegraphics[width=0.21\linewidth, clip=true, trim=440pt 50pt 2450pt 40pt]{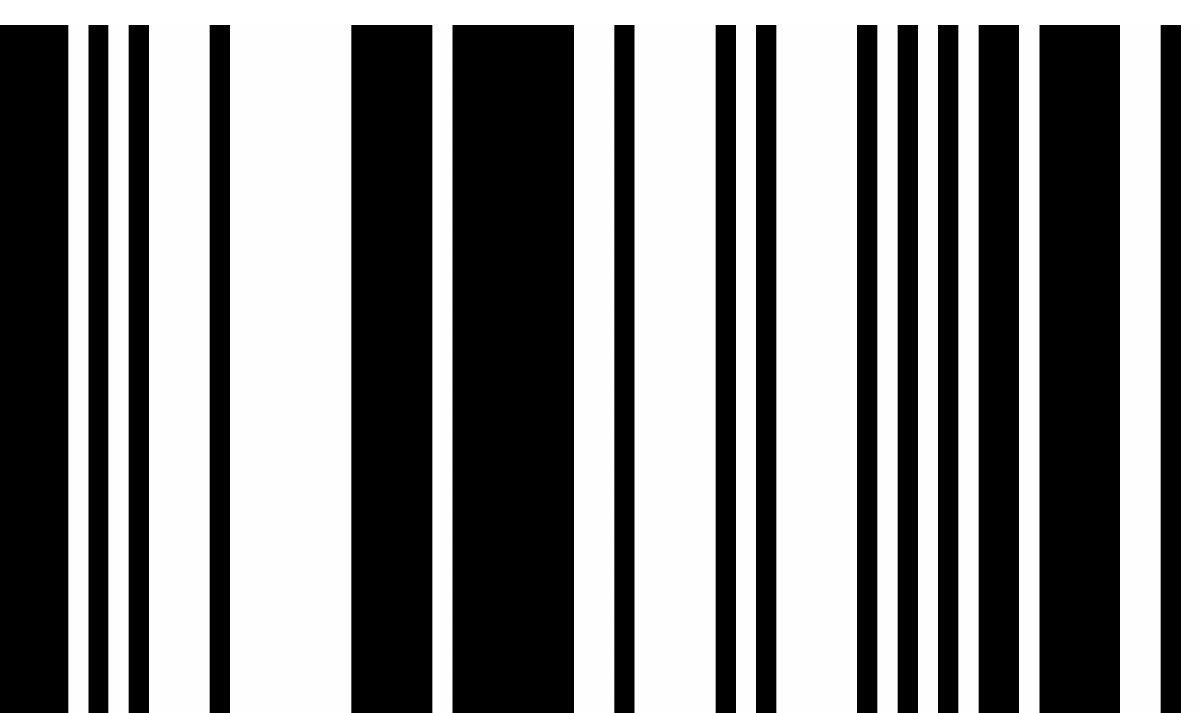}{\,\small{i)}}
\includegraphics[width=0.21\linewidth, clip=true, trim=1150pt 50pt 1750pt 40pt]{diagrams_barcode/S0_comp_tight.pdf}{\,\small{ii)}}
\includegraphics[width=0.21\linewidth, clip=true, trim=1830pt 50pt 1050pt 40pt]{diagrams_barcode/S0_comp_tight.pdf}{\,\small{iii)}}
\includegraphics[width=0.21\linewidth, clip=true, trim=2530pt 50pt 350pt 40pt]{diagrams_barcode/S0_comp_tight.pdf}{\,\small{iv)}}
\begin{minipage}{.48\textwidth}
\caption{
Barcode denoising example:\linebreak
(i) ground truth, 
(ii) noisy data
($\smash{\sigma = \sqrt{2}\cdot0.7}$),
(iii) Alg.~\ref{alg:1} 
($\lambda = 1.0$, $\rho = 0.1$) without projection,
(iv) ANISO-TV 
($\lambda = 1.4$) 
with projection $\chi_0$,
(v) fast-TV ($\lambda = 1.0$) without projection,
(vi) fast-TV ($\lambda = 1.0$) with projection $\chi_0$. 
}
\label{fig:S0_1D_comp_ADMM_ANISO}
\end{minipage}
\hspace{0.05cm}
\begin{minipage}{.50\textwidth}
    \includegraphics[width=0.42\linewidth, clip=true, trim=1830pt 50pt 1050pt 40pt]{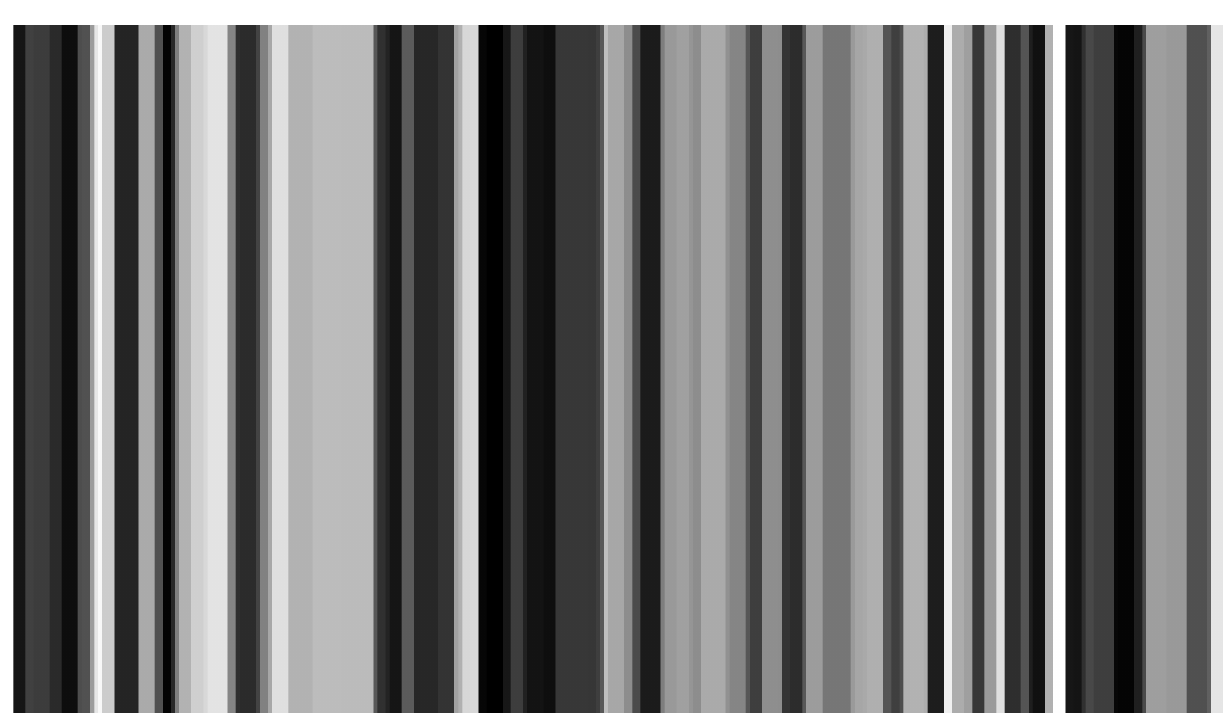}{\,\small{v)}}
    \hspace{0.15cm}
    \includegraphics[width=0.42\linewidth, clip=true, trim=2530pt 50pt 350pt 40pt]{diagrams_barcode/S0_comp_non_tight.pdf}{\,\small{vi)}}
\end{minipage}
\end{figure}

\begin{table}
\begin{minipage}{0.49\textwidth}
\caption{Averages for 50 randomly generated barcodes
for different noise levels.
One specific instance is illustrated in Fig.~\ref{fig:S0_1D_comp_ADMM_ANISO}.
}
\label{table:S0-signal}
\resizebox{\textwidth}{!}{
\begin{tabular}{@{}lllllll@{}}
\toprule
\multirow{2}{*}{\rotatebox{90}{$= \sigma$}} & \multirow{2}{*}{Algorithm} &\multicolumn{2}{c}{signal errors} & \multirow{2}{*}{$\lambda$}
& time & distance \\
		    &     & MSE  
      & MIoU &     
      & (sec.) & to sphere\\
\midrule 
\multirow{3}{*}{\rotatebox{90}{$\sqrt{2}\frac{3}{10}$}} & 
fast-TV  & \textbf{0.00323}
& \textbf{0.99652} & $1.0$ 
& \multirow{3}{*}{\rotatebox{90}{$< 0.1$}} & 0.01887 \\
& ANISO-TV 	& 0.00328 
& 0.99639 & $0.6$
&  & 0.00001\\
& Alg.~\ref{alg:1} 	& \textbf{0.00323} 
& \textbf{0.99652} & $1.0$ 
&  & \textbf{0.00000} \\
\midrule
\multirow{3}{*}{\rotatebox{90}{$\sqrt{2}\frac{5}{10}$}} & 
fast-TV & \textbf{0.00983}
& \textbf{0.98259} & $1.0$
& \multirow{3}{*}{\rotatebox{90}{$< 0.1$}} & 0.04521 \\
& ANISO-TV 	& 0.01027 
& 0.98154 & $0.7$
&   & 0.00023 \\
& Alg.~\ref{alg:1} 	& \textbf{0.00983} 
& \textbf{0.98259} & $1.3$
&   & \textbf{0.00000}\\
\midrule
\multirow{3}{*}{\rotatebox{90}{$\sqrt{2}\frac{7}{10}$}} & 
fast-TV & \textbf{0.01885} 
& \textbf{0.94201} & $1.3$ 
& \multirow{3}{*}{\rotatebox{90}{$< 0.1$}} & 0.08776 \\
& ANISO-TV    & 0.01958 
&  0.93811 & $1.0$ 
&   & 0.00443 \\
& Alg.~\ref{alg:1} 	& \textbf{0.01885}
& \textbf{0.94201} & $1.7$ 
&   & \textbf{0.00000}\\
\midrule
\multirow{3}{*}{\rotatebox{90}{$\sqrt{2}\frac{9}{10}$}} & 
fast-TV & \textbf{0.02423} 
& \textbf{0.90976} & $2.2$
& \multirow{3}{*}{\rotatebox{90}{$< 0.1$}} & 0.10067 \\
& ANISO-TV    & 0.02638
& 0.89385 & $1.2$ 
&  & 0.00644\\
& Alg.~\ref{alg:1} 	& \textbf{0.02423} 
& \textbf{0.90976} & $2.2$
&   & \textbf{0.00000}\\
\bottomrule
\end{tabular}}
\end{minipage}
\hfill
\begin{minipage}{0.49\textwidth}
\caption{Averages for 50 randomly generated QR codes for different noise levels.
One specific instance is illustrated in Fig.~\ref{fig:S0-2D}.
}
\label{table:S0-image}
\resizebox{\textwidth}{!}{\begin{tabular}{@{}lllllll@{}}
\toprule
\multirow{2}{*}{\rotatebox{90}{$= \sigma$}} & \multirow{2}{*}{Algorithm} &\multicolumn{2}{c}{signal errors} & \multirow{2}{*}{$\lambda$} ´
& time & distance\\
		    &     & MSE  
      & MIoU &    & 
      (sec.) & to sphere\\
\midrule 
\multirow{3}{*}{\rotatebox{90}{$\sqrt{2}\frac{3}{10}$}} & 
fast-TV  & \textbf{0.00008} 
& \textbf{0.99984} & $0.3$ 
& $<0.1$ & 0.05886 \\
& ANISO-TV 	& 0.00030 
& 0.99813 & $1.4$ 
& 1.7 & 0.00086\\
& Alg.~\ref{alg:1} 	& \textbf{0.00008} 
& \textbf{0.99984} & $0.3$ 
& 0.9 & \textbf{0.00000} \\
\midrule
\multirow{3}{*}{\rotatebox{90}{$\sqrt{2}\frac{5}{10}$}} & 
fast-TV & \textbf{0.00036} 
& \textbf{0.99734} & $0.7$ 
& $<0.1$ & 0.07879 \\
& ANISO-TV 	& 0.00112 
& 0.97495 & $1.9$
& 1.8 & 0.00122\\
& Alg.~\ref{alg:1} 	& \textbf{0.00036} 
& \textbf{0.99734} & $0.7$ 
& 1.1 & \textbf{0.00000} \\
\midrule
\multirow{3}{*}{\rotatebox{90}{$\sqrt{2}\frac{7}{10}$}} & 
fast-TV  & 0.00069 
& 0.99018 & $1.0$
& $<0.1$ & 0.10021\\
& ANISO-TV 	& 0.00190 
& 0.92942 & $2.4$
& 1.9 & 0.00732 \\
& Alg.~\ref{alg:1} 	& \textbf{0.00069} 
& \textbf{0.99030} & $1.0$ 
&  1.1 & \textbf{0.00000} \\
\midrule
\multirow{3}{*}{\rotatebox{90}{$\sqrt{2}\frac{9}{10}$}} & 
fast-TV & 0.00106 
& 0.97741 & $1.7$ 
& $<0.1$ & 0.12263 \\
& ANISO-TV    & 0.00263 
&  0.86956 & $2.9$
&  2.1 & 0.01445 \\
& Alg.~\ref{alg:1} 	& \textbf{0.00106} 
& \textbf{0.97753} & $1.7$ 
& 1.2 & \textbf{0.00000} \\
\bottomrule
\end{tabular}}
\end{minipage}
\end{table}

\begin{figure}[t]
\includegraphics[width=0.30\linewidth, clip=true, trim=160pt 40pt 1020pt 55pt]{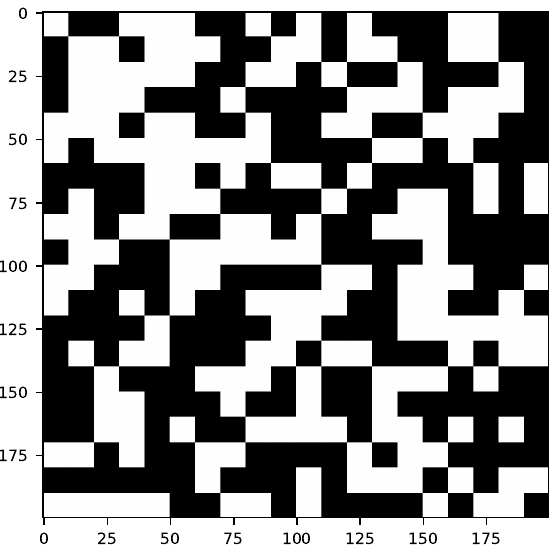}\raisebox{8pt}{\,\small{i)}}
%\hspace{0.1cm}
\includegraphics[width=0.30\linewidth, clip=true, trim=452pt 40pt 728pt 55pt]{diagrams_barcode/0D_2D_comp_tight.pdf}\raisebox{8pt}{\,\small{ii)}}
%\hspace{0.1cm}
\includegraphics[width=0.30\linewidth, clip=true, trim=745pt 40pt 435pt 55pt]{diagrams_barcode/0D_2D_comp_tight.pdf}\raisebox{8pt}{\,\small{iii)}}
\includegraphics[width=0.30\linewidth, clip=true, trim=1036pt 40pt 143pt 55pt]{diagrams_barcode/0D_2D_comp_tight.pdf}\raisebox{8pt}{\,\small{iv)}}
%\hspace{0.1cm}
\includegraphics[width=0.30\linewidth, clip=true, trim=745pt 40pt 435pt 55pt]{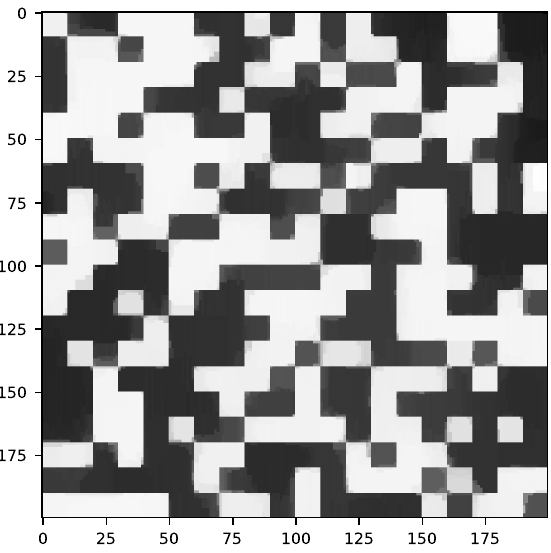}\raisebox{8pt}{\,\small{v)}}
%\hspace{0.1cm}
\includegraphics[width=0.30\linewidth, clip=true, trim=1036pt 40pt 143pt 55pt]{diagrams_barcode/0D_2D_comp_non_tight.pdf}\raisebox{8pt}{\,\small{vi)}}

\caption{
QR code denoising example: 
(i) ground truth,
(ii) noisy data %obtained by standard deviation 
($\sigma = \sqrt{2}\cdot 0.5$),
(iii) Alg.~\ref{alg:1} 
($\lambda = 1.0$, $\rho = 0.1$) without projection,
%(convergence for different noise levels in less the 30 iterations);
(iv) ANISO-TV 
($\lambda = 1.6$) 
with projection $\chi_0$,
(v) fast-TV 
($\lambda = 1.0$) without projection,
(vi) fast-TV 
($\lambda = 1.0$) with projection $\chi_0$.
% - heuristic.
}
\label{fig:S0-2D}
\end{figure}

\paragraph{Circle-Valued Signal Denoising}

For $\sphere_1$-valued signals,
we compare Alg.~\ref{alg:1}
with (i) the Cylic Proximal Point Algorithm for first order differences (CPPA-TV) from \cite{BerLauSteWei18}
and (ii) the previous fast-TV heuristic,
where the fast TV program \cite{Con12v2,Con13v4} is applied coordinatewise,
and the result is orthogonally projected onto $ \sphere_1^N$ afterwards.
The TV regularizer for CPPA-TV is based on geodesic distances
whereas our TV term is based on the Manhattan norm in $\R^2$.
For the one-dimensional setting,
we rely on the ground truth from \cite[Sec.~5.1]{BerLauSteWei14} 
disturbed by wrapped Gaussian noise with standard deviation $\sigma$,
and, for the two-dimensional setting,
on the ground truth from \cite{Ken23}
disturbed by von~Mises--Fischer noise with capacity $\kappa$.
We stop CPPA-TV as soon as 
the residuum (the mean difference between subsequent iterates) is smaller than $10^{-4}$ 
and Alg.~\ref{alg:1} as soon as
the residuum is smaller than $10^{-6}$
or the distance to the circle is smaller than $10^{-5}$.
Similar to the previous paragraph,
Figure~\ref{fig:S1-signal} and \ref{fig:S1-2D-toy}
show qualitative denoising results and
Table~\ref{table:S1-signal} and \ref{table:S1-image}
contain quantitative studies.
Compared to CPPA-TV,
Alg.~\ref{alg:1} shows a significantly smaller computation time
and an improved denoising effect.
Fast-TV is again a rapid alternative
heuristically yielding comparable results as Alg.~\ref{alg:1}.
Similarly to \cite{BeBrSt23},
a real-world application for $\sphere_1$-valued data
is the hue denoising with respect to the HSV color space.
In Figure~\ref{fig:S1-2D-hue},
we repeat the experiment in \cite{BeBrSt23}
to denoise the color of a coral,
however,
using the TV instead of the Tikhonov regularization.
In all experiments,
Alg.~\ref{alg:1} directly returns an $\sphere_1$-valued 
and therefore global solution of \eqref{eq:nonsmooth_tik}
making the final projection again obsolete.

\begin{figure}[t]
\includegraphics[width=\linewidth, clip=true, trim=90pt 20pt 80pt 40pt]{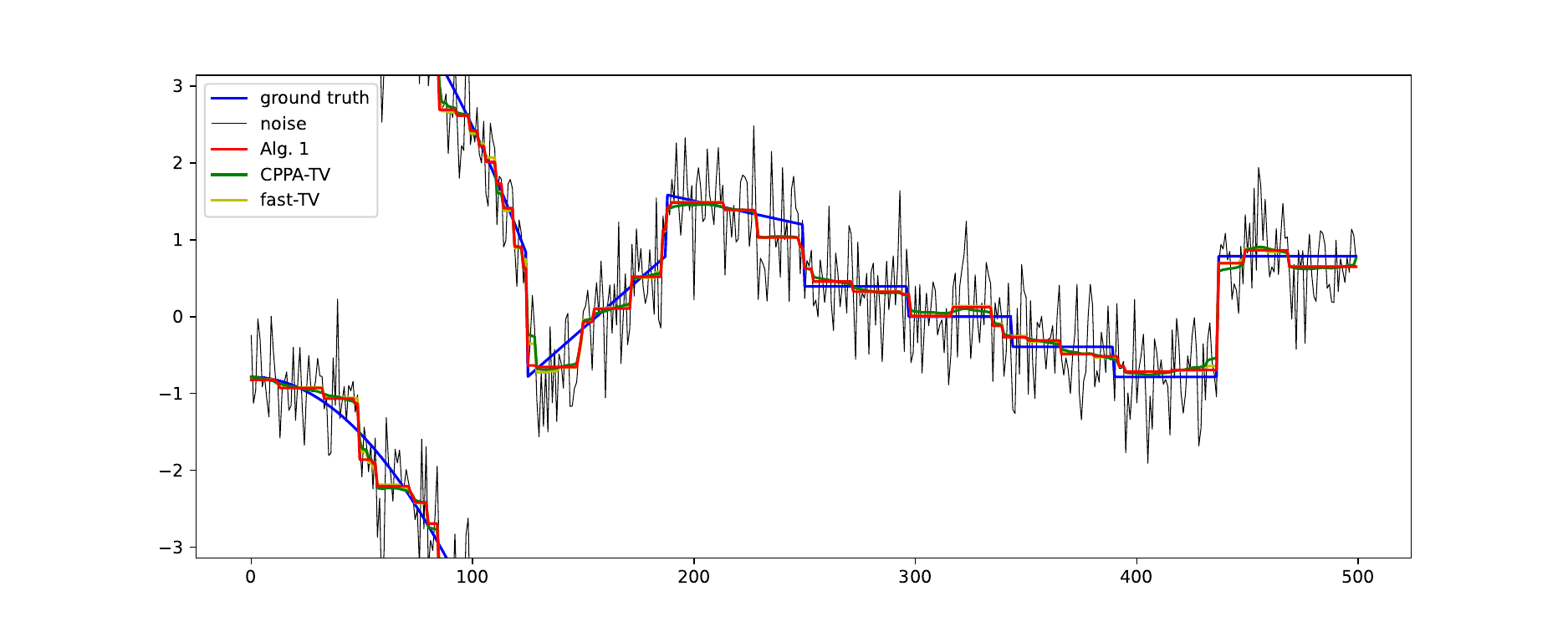}
\caption{Circle-valued signal denoising 
    with ground truth from \cite{BerLauSteWei18} %\cite{BerLauSteWei14, BerLauSteWei18}
    and wrapped Gaussian noise with standard deviation $\sigma \smash{=\frac{1}{2}}$.
    Solutions calculated by:
    (red) Alg.~\ref{alg:1} ($\lambda = 1.2$, $\rho = 1$) without projection, 
    (green) CPPA-TV ($\lambda = 1.8$, $\lambda_0 = 10$) without projection,
    and 
    (yellow) fast-TV heuristic ($\lambda = 1.0$) with final projection.}
    \label{fig:S1-signal}
%\end{minipage}
\end{figure}

\begin{figure}[t!]
\includegraphics[width=\linewidth, clip=true, trim=140pt 20pt 130pt 40pt]{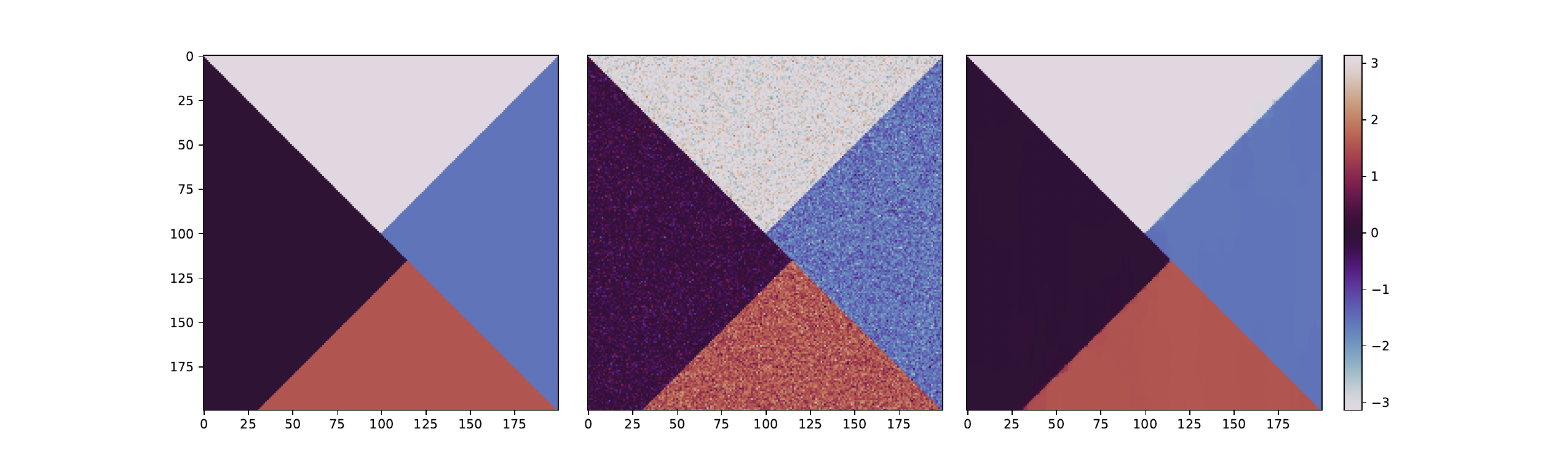}
\caption {
Toy-data example following \cite{Ken23} for $\sphere_1$-image denoising
(from left to right):
\linebreak
(i) ground truth,
(ii) noisy measurement generated 
by the von Mises--Fisher distribution with $\kappa = 10$,
(iii) solution via Alg.~\ref{alg:1} ($\lambda = 0.55$, $\rho = 10$)
without final projection.
}
\label{fig:S1-2D-toy}   
\end{figure}

\begin{table}[t!]
\begin{minipage}{.48\textwidth}
\caption{
Averages for 20 randomly generated noisy instances of the ground truth in Fig.~\ref{fig:S1-signal}
for different noise levels.
}
\label{table:S1-signal}
\resizebox{\textwidth}{!}{
\begin{tabular}{@{}lllllllll@{}}
\toprule
\multirow{2}{*}{\rotatebox{90}{$= \sigma$}} & \multirow{2}{*}{Algorithm} &\multicolumn{1}{c}{signal error} & \multirow{2}{*}{$\lambda$} & time & distance\\
		    &     & MSE 
      &   & (sec.) & to sphere\\
\midrule 
\multirow{3}{*}{\rotatebox{90}{$\frac{1}{5}$}} & 
CPPA-TV 	& 0.0052151
& $0.9$ & $8.5$ & ---\\
& fast-TV 	& 0.0051141 
& $0.6$ & $<0.1$ & $0.00125$\\
& Alg.~\ref{alg:1} 	& \textbf{0.0049986} 
& $0.7$ & $0.2$ & \textbf{0.00000}\\
\midrule
\multirow{3}{*}{\rotatebox{90}{$\frac{1}{4}$}} & 
CPPA-TV 	& 0.0072028 
& $1.3$ & $8.6$ & ---\\
& fast-TV 	& 0.0071249 
& $0.7$ & $<0.1$ & $0.00633$\\
& Alg.~\ref{alg:1} 	& \textbf{0.0070003} 
& $0.9$ & $0.28$ & \textbf{0.00000}\\
\midrule
\multirow{3}{*}{\rotatebox{90}{$\frac{1}{3}$}} & 
CPPA-TV 	& 0.0116987 
& $1.7$ & $8.9$ & ---\\
& fast-TV 	& 0.0117756 
& $0.7$ & $<0.1$ & $0.00854$\\
& Alg.~\ref{alg:1} 	& \textbf{0.0114888} 
& $1.2$ & $0.35$ & \textbf{0.00000}\\
\midrule
\multirow{3}{*}{\rotatebox{90}{$\frac{1}{2}$}} & 
CPPA-TV 	& 0.0215876
& $2.0$ & $9.8$ & --- \\
& fast-TV 	& 0.0215188 
& $0.9$ & $<0.1$ & $0.01805$\\
& Alg.~\ref{alg:1} 	& \textbf{0.0214925}	
& $1.4$ & $0.48$ & \textbf{0.00000}\\
\bottomrule
\end{tabular}}
\end{minipage}
\hfill
\begin{minipage}{0.48\textwidth}
\caption{
Averages for 20 randomly generated noisy instances of the ground truth in Fig.~\ref{fig:S1-2D-toy}
for different noise levels. 
}
\label{table:S1-image}
\resizebox{\textwidth}{!}{
\begin{tabular}{@{}llllll@{}}
\toprule
\multirow{2}{*}{\rotatebox{90}{$= \kappa$}} & \multirow{2}{*}{Algorithm} &\multicolumn{1}{c}{signal error} & \multirow{2}{*}{$\lambda$} & time & distance\\
		    &     & MSE  %& MAE 
      &    & (sec.) & to sphere\\
\midrule 
\multirow{3}{*}{\rotatebox{90}{$50$}} & 
CPPA-TV 	& 0.00076 %& 0.01268 
& $0.15$ & $128.2$ & --- \\
& fast-TV 	& 0.00556 %& 0.01982 
& $0.25$ & $<0.1$ & 0.00055\\
& Alg.~\ref{alg:1} 	& \textbf{0.00043} %& \textbf{0.00760} 
& $0.2$ & $4.3$ & \textbf{0.00000}\\
\midrule 
\multirow{3}{*}{\rotatebox{90}{$20$}} & 
CPPA-TV 	& 0.00198 %& 0.01982 
& $0.25$ & $166.1$ & --- \\
& fast-TV 	& 0.00199 %& 0.01982 
& $0.25$ & $<0.1$ & 0.00098 \\
& Alg.~\ref{alg:1} 	& \textbf{0.00107} %& \textbf{0.01256} 
& $0.3$ & $6.2$ & \textbf{0.00000}\\
\midrule
\multirow{3}{*}{\rotatebox{90}{$10$}} & 
CPPA-TV 	& 0.00409 %& 0.02820 
& $0.55$ & $191.7$ & --- \\
& fast-TV 	& 0.00388 %& 0.01982 
& $0.25$ & $<0.1$ & 0.00155\\
& Alg.~\ref{alg:1} 	& \textbf{0.00218} %& \textbf{0.01819} 
& $0.45$ & $10.5$ & \textbf{0.00000}\\
\bottomrule
\end{tabular}}
\end{minipage}    
\end{table}

\begin{figure}[t!]
\includegraphics[width=\linewidth, clip=true, trim=140pt 20pt 130pt 40pt]{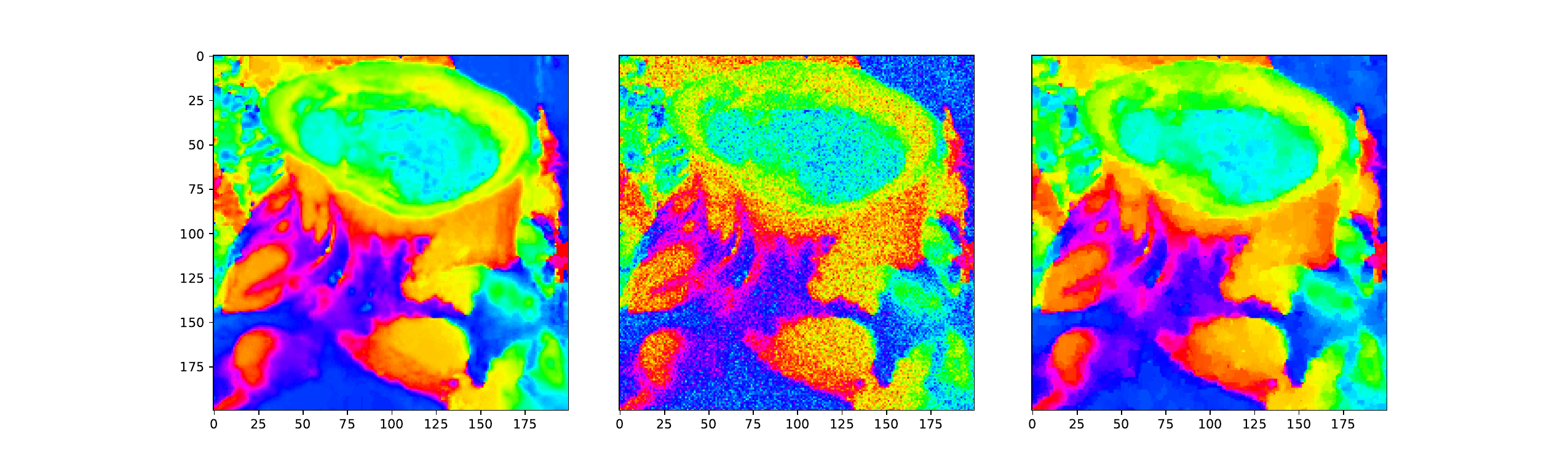}
\vspace{-0.5cm}
\caption {
Hue-denoising for the coral image from \cite{BeBrSt23} in the HSV color space
(from left to right): 
(i) ground truth, % $\Vx_{\text{true}}$, 
(ii) noisy measurements generated by the von~Mises--Fisher distribution with $\kappa~=~10$, 
(iii) solution via Alg.~\ref{alg:1} ($\lambda~=~0.2$, $\rho~=~100$) without final projection.
The computation takes less than 3 seconds
with MSE $9.843\cdot10^{-4}$ 
and averaged distance to the sphere $2.5\cdot10^{-6}$.
}
\label{fig:S1-2D-hue}
\end{figure}

\begin{figure}[t!]
\includegraphics[width=\linewidth, clip=true, trim=140pt 20pt 130pt 40pt]{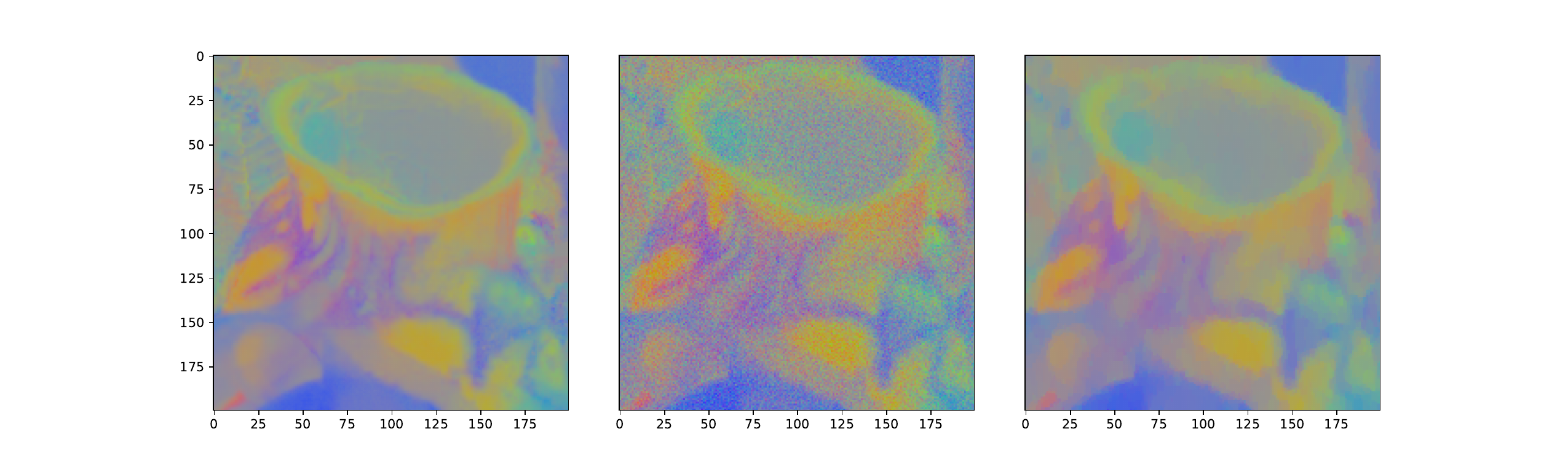}
\vspace{-0.5cm}
\caption { 
Chromaticity-denoising for the coral image from \cite{BeBrSt23}
in the RGB color space (from left to right):
(i) ground truth, % $\Vx_{\text{true}}$, 
(ii) noisy measurement generated by the von~Mises--Fisher distribution with capacity $\kappa~=~200$,
(iii) solution via Alg.~\ref{alg:1} 
($\lambda~=~0.06$, $\rho~=~100$)
without final projection 
and with MSE $9.46\cdot10^{-4}$
and averaged distance to the sphere $1.34\cdot 10^{-6}$.
}
\label{fig:S2-2D-chromaticity}
\end{figure}

\paragraph{Sphere- and SO(3)-Valued Signal Denoising}

Looking at the chromaticity of color image,
we naturally have to deal with $\sphere_2$-valued data. 
More precisely,
the chromaticity is defined as the normalized RGB vector.
Similarly to the hue denoising above,
we apply our method to denoise the chromaticity 
of the coral image from \cite{BeBrSt23}.
The proof-of-concept result is shown in Figure~\ref{fig:S2-2D-chromaticity},
where we stop the algorithm as soon as the residuum is at most $10^{-6}$ 
or the distance to the sphere is at most $10^{-4}$.
Denoising the chromaticity of the whole image with Alg.~\ref{alg:1}
here requires merely 8 seconds
and immediately yields a $\sphere_2$-valued solution.

Finally,
we apply Alg.~\ref{alg:1} for denoising SO(3)-valued data. 
This kind of data occurs for example in electron backscatter tomography \cite{BHS11,BHJPSW10}.
The crucial idea is here the representation 
of a 3d rotation matrix as unit quaternion 
and thus as unit vector in $\R^4$ up to sign. 
More precisely,
the $\sphere_3$ forms a double cover of SO(3)
\cite[Ch.~III, Sec.~10]{Bre93}.
In analogy to \cite[Sec.~4]{BeBrSt23},
Figure~\ref{fig:SO(3)-2D-synthetic} shows an SO(3) denoising task 
for appropriate cartoon-like toy data.
Using the same stopping criteria as before,
the denoising takes 12 seconds.
Numerically,
we observe convergence to a sphere-valued solution,
which can be immediately interpreted as SO(3)-valued image.

\begin{figure}[t!]
\includegraphics[width=\linewidth, clip=true, trim=260pt 120pt 230pt 120pt]{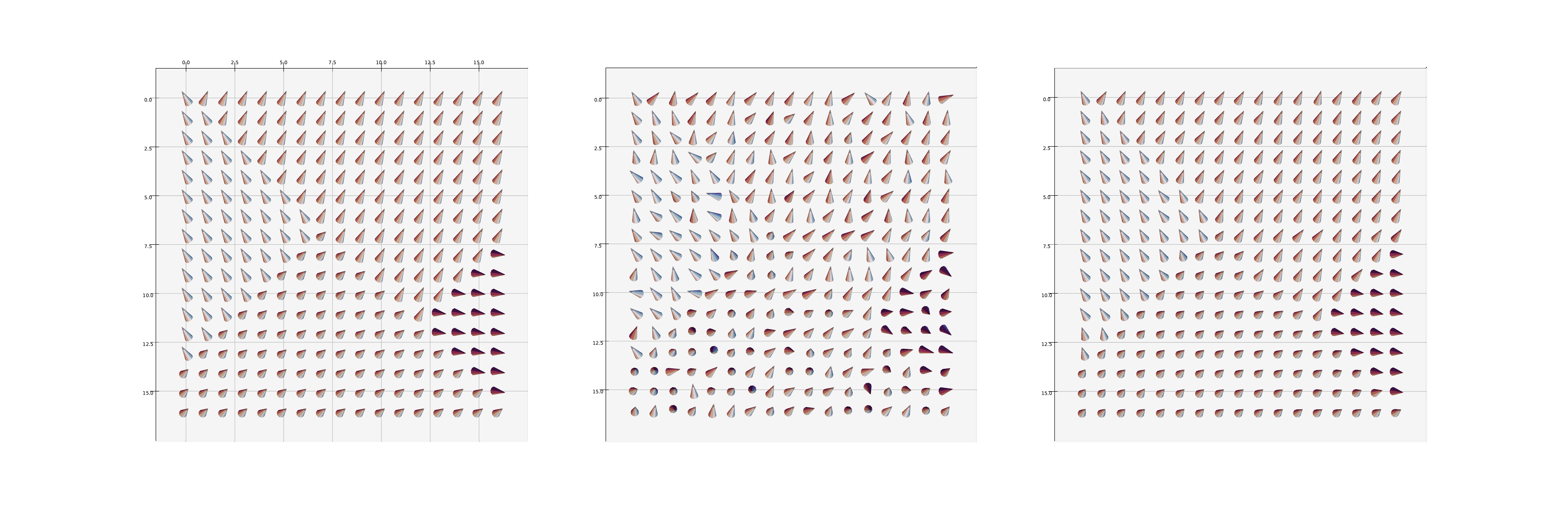}
\vspace{-0.5cm}
\caption{
Denoising of a synthetic SO(3)-image 
($90 \times 90$ pixels)
inspired by \cite{Ken23}, 
following \cite{BeBrSt23} (from left to right, downsampled results): %\linebreak
(i) ground truth, %$\Vx_{\text{true}}$,
(ii) noisy measurment generated by the von~Mises--Fisher distribution 
with capacity $\kappa_1~=~10$ for the rotation angles
and $\kappa_2~=~10$ for the rotation axis.
(iii) solution via Alg.~\ref{alg:1} 
($\lambda~=~0.10$, $\rho~=~100$)
without final projection
and with MSE $5.889\cdot10^{-3}$
and averaged distance to the unit quaternions $1.45\cdot 10^{-7}$.
}
\label{fig:SO(3)-2D-synthetic}
\end{figure}

\bibliographystyle{abbrv}
\bibliography{literatur}

\begin{thebibliography}{10}

\bibitem{ASWK1993}
B.~L. Adams, S.~I. Wright, and K.~Kunze.
\newblock Orientation imaging: the emergence of a new microscopy.
\newblock {\em Metall. Mater. Trans. A}, 24:819--831, 1993.

\bibitem{BBSW16}
M.~Ba{\v c}{\'a}k, R.~Bergmann, G.~Steidl, and A.~Weinmann.
\newblock A second order non-smooth variational model for restoring manifold-valued images.
\newblock {\em SIAM J. Sci. Comput.}, 38(1):567--597, 2016.

\bibitem{BHJPSW10}
F.~Bachmann, R.~Hielscher, P.~E. Jupp, W.~Pantleon, H.~Schaeben, and E.~Wegert.
\newblock Inferential statistics of electron backscatter diffraction data from within individual crystalline grains.
\newblock {\em J. Appl. Crystallogr.}, 43:1338--1355, 2010.

\bibitem{BHS11}
F.~Bachmann, R.~Hielscher, and H.~Schaeben.
\newblock Grain detection from 2d and 3d {EBSD} data—specification of the {MTEX} algorithm.
\newblock {\em Ultramicroscopy}, 111(12):1720--1733, 2011.

\bibitem{bauschke}
H.~H. Bauschke and P.~L. Combettes.
\newblock {\em Convex analysis and monotone operator theory in {H}ilbert spaces}.
\newblock Springer, New York, 2011.

\bibitem{BeBrSt23}
R.~Beinert, J.~Bresch, and G.~Steidl.
\newblock {arXiv:2307.10980}, 2023.

\bibitem{BerChaHiePerSte16}
R.~Bergmann, R.~H. Chan, R.~Hielscher, J.~Persch, and G.~Steidl.
\newblock Restoration of manifold-valued images by half-quadratic minimization.
\newblock {\em Inverse Probl. Imaging}, 10(2):281--304, 2016.

\bibitem{BerLauSteWei18}
R.~Bergmann, F.~Laus, G.~Steidl, and A.~Weinmann.
\newblock Second order differences of cyclic data and applications in variational denoising.
\newblock {\em {SIAM} J. Imaging Sci.}, 7(4):2916--2953, jan 2014.

\bibitem{BerLauSteWei14}
R.~Bergmann, F.~Laus, G.~Steidl, and A.~Weinmann.
\newblock Second order differences of cyclic data and applications in variational denoising.
\newblock {\em SIAM J. Imaging Sci.}, 7(4):2916--2953, 2014.

\bibitem{Bre93}
G.~E. Bredon.
\newblock {\em Topology and Geometry}.
\newblock Springer, New York, 1993.

\bibitem{BRFa2000}
R.~B\"urgmann, P.~A. Rosen, and E.~J. Fielding.
\newblock Synthetic aperture radar interferometry to measure earth's surface topography and its deformation.
\newblock {\em Annu. Rev. Earth Planet Sci.}, 28(1):169--209, 2000.

\bibitem{CHOGENOB2011}
R.~Choksi, Y.~van Gennip, and A.~Oberman.
\newblock {A}nisotropic {T}otal {V}ariation {R}egularized ${L}^1$-{A}pproximation and {D}enoising/{D}eblurring of 2{D} {B}ar {C}odes.
\newblock {\em Inverse Probl. Imaging}, 5(3):591--617, 2011.

\bibitem{Con12v2}
L.~Condat.
\newblock {hal:00675043}.
\newblock 2012.

\bibitem{Con13v4}
L.~Condat.
\newblock {A Direct Algorithm for 1D Total Variation Denoising}.
\newblock {\em {IEEE Signal Process. Lett.}}, 20(11):1054--1057, 2013.

\bibitem{condat_1D2D}
L.~Condat.
\newblock Tikhonov regularization of circle-valued signals.
\newblock {\em IEEE Trans. Signal Process.}, 70:2775--2782, 2022.

\bibitem{CS13}
D.~Cremers and E.~Strekalovskiy.
\newblock Total cyclic variation and generalizations.
\newblock {\em J. Math. Imaging Vis.}, 47(3):258--277, 2013.

\bibitem{Fed59}
H.~Federer.
\newblock Curvature measures.
\newblock {\em Trans. Am. Math. Soc.}, 93(3):418--491, 1959.

\bibitem{FleRis60}
W.~H. Fleming and R.~Rishel.
\newblock An integral formula for total gradient variation.
\newblock {\em Arch. Math.}, 11(1):218--222, 1960.

\bibitem{GS14}
P.~Grohs and M.~Sprecher.
\newblock Total variation regularization on {R}iemannian manifolds by iteratively reweighted minimization.
\newblock {\em Inf. Inference}, 5(4):353--378, 2016.

\bibitem{Ken23}
R.~Kenis, E.~Laude, and P.~Patrinos.
\newblock {arXiv:2308.00079}, 2023.

\bibitem{LEHS2008}
T.~Lan, D.~Erdogmus, S.~J. Hayflick, and J.~U. Szumowski.
\newblock In {\em Proceedings MLSP '08, Cancun, Mexico}, pages 239--243, New York, 2008. IEEE.

\bibitem{LNPS2017}
F.~Laus, M.~Nikolova, J.~Persch, and G.~Steidl.
\newblock A nonlocal denoising algorithm for manifold-valued images using second order statistics.
\newblock {\em SIAM J. Imaging Sci.}, 10(1):416--448, 2017.

\bibitem{LSKC13}
J.~Lellmann, E.~Strekalovskiy, S.~Koetter, and D.~Cremers.
\newblock In {\em Proceedings ICCV '13, Sydney, Australia}, pages 2944--2951, New York, 2013. IEEE.

\bibitem{NikSte14}
M.~Nikolova and G.~Steidl.
\newblock Fast hue and range preserving histogram specification: theory and new algorithms for color image enhancement.
\newblock {\em IEEE Trans. Image Process.}, 23(9):4087--4100, 2014.

\bibitem{PPS2017}
J.~Persch, F.~Pierre, and G.~Steidl.
\newblock Exemplar-based face colorization using image morphing.
\newblock {\em J. Imaging}, 3(4):48, 2017.

\bibitem{QKL2010}
M.~H. Quang, S.~H. Kang, and T.~M. Le.
\newblock Image and video colorization using vector-valued reproducing kernel {H}ilbert spaces.
\newblock {\em J. Math. Imaging Vis.}, 37:49--65, 2010.

\bibitem{WDS14}
A.~Weinmann, L.~Demaret, and M.~Storath.
\newblock Total variation regularization for manifold-valued data.
\newblock {\em SIAM J. Imaging Sci.}, 7(4):2226--2257, 2014.

\end{thebibliography}

\end{document}